\documentclass[11pt,twoside]{article}

\usepackage{mathrsfs,amsfonts,amsmath,amssymb}
\usepackage{latexsym}
\usepackage{comment}

\newtheorem{satz}{Theorem}

\newtheorem{theorem}[satz]{Theorem}
\newtheorem{lemma}[satz]{Lemma}

\newtheorem{corollary}[satz]{Corollary}
\newtheorem{remark}[satz]{Remark}

\def\Z{\mathbb {Z}}
\def\F{\mathbb {F}}
\def\E{\mathsf{E}}

\def\a{\alpha}

\def\d{\delta}

\def\({\big (}
\def\){\big )}

\def\le{\leqslant}
\def\ge{\geqslant}
\def\_phi{\varphi}
\def\eps{\varepsilon}

\def\Gr{{\mathbf G}}
\def\FF{\widehat}

\def\ov{\overline}

\def\la{\lambda}
\def\D{\Delta}

\newcommand{\bp}{\bigskip}

\author{I.D. Shkredov}
\title{
On the 
distribution of quadratic residues 
}
\date{}
\begin{document}
	\maketitle


\begin{center}
	Annotation.
\end{center}

{\it \small
    In our paper, we apply additive--combinatorial methods to study the distribution of the set of squares $\mathcal{R}$ in the prime field. We obtain the best   
    upper bound on the number of gaps in  $\mathcal{R}$  at the moment and generalize this result for sets with small doubling. 
}
\\

\section{Introduction}

Let $p$ be a prime number. 
Consider the simplest non--trivial multiplicative subgroup in $\F_p^* = \F_p \setminus \{0\}$, namely, the set of quadratic residues 
$$
    \mathcal{R} = \{ x^2  ~:~ x\in \F^*_p \} \,.  
$$
The set $\mathcal{R}$ (and its complement $\mathcal{N} = \F_p^* \setminus \mathcal{R}$, i.e. the set of quadratic non--residues) is a good constructive model for a randomly chosen  subset of $\F_p^*$ (each element of the set is taken with probability $1/2$).  
On the other hand, 
the set of quadratic residues is 
strongly 
related to 
bounds on $L$--functions, see 
\cite{Burgess2}, \cite{Burgess3}, \cite{Burgess1} 
and \cite{IK_book}. 
The question 
regarding 
the distribution of the elements of $\mathcal{R}$ attracts   the attention of many researchers see, e.g., \cite{alsetri2022hilbert}, 
\cite{Burgess2}---\cite{chang2008question}, 
\cite{FI_char_sums}, \cite{H-B_Burgess_moments}, \cite{Karatsuba_survey}, \cite{Polya}, 
\cite{shao2015character}---\cite{volostnov2018double} 
and so on. 
For example, Vinogradov \cite{Vinogradov_selected} considered the maximal distance $d(p)$ between quadratic residues and conjectured that  $d(p) \ll_\eps p^{\eps}$, where $\eps>0$ is any number.
The first non--trivial results in this direction were  obtained in \cite{Polya} and  \cite{Vinogradov_1918}. 
The best result at the moment is due to Burgess (see \cite{Burgess2})
\begin{equation}\label{f:Burgess}
    d(p) \ll p^{1/4} \log p \,.
\end{equation}
Another question concerning the distribution of the set $\mathcal{R}$ is the following.
Take any  integer $1\le h\le d(p)$ and define 
\begin{equation}\label{def:S(h)}
    S(h) = S^{\mathcal{R}} (h)= \max \{ |S| ~:~ S\subset \F_p,\, S\dotplus \{ 0,1,\dots,h-1\} \subseteq \mathcal{R} \} \,.
\end{equation}
What can be said about the quantity $S(h)$? In other words,  how many segments of length $h$ are contained in $\mathcal{R}$ (equivalently, how many gaps of $h$ consecutive elements exist in $\mathcal{N}$)? 
The mentioned 
problem was considered in \cite[Theorem 2']{FI_char_sums}, \cite[Theorem 1]{H-B_Burgess_moments}  and the best result in this direction can be found in paper \cite{shao2015character}, where the author proves, in particular, that 
\begin{equation}\label{f:Shao}
    S(h) \ll_{\eps,r} \frac{p^{1/2 + 1/2r +\eps}}{h^2} \,,
\end{equation}
where $\eps>0$ is an arbitrary number, $r$ is a positive integer and $h\ge p^{1/2r}$. 
As for applications of the upper bounds for $S(h)$, 
refer to, say, 
\cite{alsetri2022hilbert}, \cite{DES_digits} and \cite{DES_Hamming}.

The aim of this work is to show that using combinatorial methods one can improve bound \eqref{f:Shao}. 
Also, we interpret the appearance of the multiplicative energy in Burgess's proof 
quite transparently. 
Our method is more flexible and slightly different from the classical approach and its variations, see \cite{Burgess1} and 
\cite{IK_book}. 
In 
the 
proof, we need the only consequence of the Weil bound on the multiplicative character, namely, that for any different non--zero  shifts $s_1,\dots,s_k \in \F_p$ one has 
\begin{equation}\label{f:R_s}
    |\mathcal{R}^{+}_{s_1,\dots,s_k}| := |\mathcal{R} \cap (\mathcal{R}-s_1) \cap \dots \cap (\mathcal{R}-s_k)| \le \frac{p}{2^{k+1}} + k\sqrt{p} 
\end{equation}
and similar for $\mathcal{N}$. 
Let us formulate our main result. 
For $a\neq b$ denote by $\mathcal{Q}_{a,b}$ the set 
\[
    \mathcal{Q}_{a,b} = \left\{ x \in \F_p ~:~ \frac{x+b}{x+a} \in \mathcal{R} \right\} = \frac{a-b}{\mathcal{R}-1}-b \,,
\]
and 
\[
    \mathcal{Q}'_{a,b} = \left\{ x \in \F_p ~:~ \frac{x+b}{x+a} \in \mathcal{N} \right\} = \frac{a-b}{\mathcal{N}-1}-b \,.
\]

\begin{theorem}
    Let $p$ be a prime number, $C>0$ be a real number, and $h$ be any positive integer such that 
\begin{equation}\label{cond:h}
    h\ge 2\log p - 4 \log (C\log p) + 4 \,.
\end{equation}
    Then 
\begin{equation}\label{f:distribution}
    S^{\mathcal{R}} (h),\, S^{\mathcal{N}} (h),\, S^{\mathcal{Q}_{a,a+1}} (h),\,
    S^{\mathcal{Q}'_{a,a+1}} (2h) \ll \frac{C_* \sqrt{p} \cdot \log^2 p \cdot \log^2  \left(\frac{4h}{\min^{} \{h, 2\log p\}} \right) \cdot \d_p (h)}{h^2} \,,
\end{equation}
    where $C_* = C$ if $h\le 2\log p$ and $C_* =1$ otherwise, and 
\begin{equation*}
    \d_p (h) = 
    \left\{
    \begin{array}{lr}
        1 \,, \quad  \quad \quad \,\, h\le 2 \log p \,, \\
        \frac{h}{\log p}\,, \quad \quad \, 2\log p < h \le \log^2 p \,,\\
        \log p \,, \quad \quad  h> \log^2 p \,.
    \end{array}\right.
\end{equation*}
\label{t:distribution}
\end{theorem}

It is widely believed that the Burgess bound \eqref{f:Burgess} can be improved.
On the contrary, 
it can be shown 
that condition \eqref{cond:h}  and estimate \eqref{f:distribution} are close to optimal, at least for small $h$, see Remark \ref{r:optimality} below. 
Also, it is easy to see that bound \eqref{f:distribution}, in a sense, allows us to break the square--root barrier  for $h\ge \log^{3+\eps} p$, $\eps>0$, see Corollary \ref{c:R_h}. 
We combine the obtained results concerning the square--root barrier in the next theorem.

\begin{theorem}
    Let $p$ be a prime number and $h\ge 2\log p$ be an integer. 
    Then in the notation of Theorem \ref{t:distribution} one has 
\begin{equation}\label{f:R_h_intr}
    |\mathcal{R}^{+}_{1,\dots,h}| \ll \frac{\sqrt{p}\cdot \log^2 p  \cdot \log^2 \left(\frac{4h}{\log p}\right) \cdot \d_p (h)}{h} \,.
\end{equation}
    Now let $X\subseteq \F_p$ be a set and $H \subseteq \F_p$ be an arithmetic progression. 
    Suppose that $|X| \ge \log p$ and $H+X \subseteq \mathcal{R}$. 
    Then 
\begin{equation}\label{f:R_R*_intr}
    |\mathcal{R}^{+}_{(H+X) \cup (H+X)^{-1}}| \ll \frac{\sqrt{p} \log p }{|H|^c} \,,
\end{equation}
    where $c>0$ is an absolute constant. 
\label{t:A_s_intr}
\end{theorem}

In our previous paper \cite[Section 5]{sh_BG} we obtained similar bounds to \eqref{f:R_h_intr}, \eqref{f:R_R*_intr} starting from an arbitrary family of shifts $\{s_1,\dots,s_k\}$ in the set $\mathcal{R}^{+}_{s_1,\dots,s_k}$. 
Here our 
shift sets 
are more 
specific 
(which is both an advantage and a disadvantage).

The paper is organized as follows. 
We obtain Theorem \ref{t:distribution} and the first part of Theorem \ref{t:A_s_intr} in Section \ref{sec:proof}. Also, we consider the question about maximal possible shifts of so--called generalized arithmetic progressions, see the required definitions in Section \ref{sec:def}. 
In Section \ref{sec:AC} we study additive--combinatorial generalizations 
of 
$S(h)$, which are interesting in its own right and obtain the second part of Theorem \ref{t:A_s_intr}.

\bp 

We thank Jozsef Solymosi 
for useful 
discussions. 

\section{Definitions}
\label{sec:def}

Given two sets $A,B\subseteq \F_p$, define  
the {\it sumset} 
of $A$ and $B$ as 
$$A+B:=\{a+b ~:~ a\in{A},\,b\in{B}\}\,.$$
In a similar way we define the {\it difference sets}, the {\it product set} and the {\it higher sumsets}, e.g., $2A-A$ is $A+A-A$.
If the sum of $A$ and $B$ is direct, that is $|A+B|=|A||B|$, then we write $A\dotplus B$. 
For an abelian group $\Gr$
the Pl\"unnecke--Ruzsa inequality (see, e.g., \cite{TV}) holds stating
\begin{equation}\label{f:Pl-R} 
|nA-mA| \le \left( \frac{|A+A|}{|A|} \right)^{n+m} \cdot |A| \,,
\end{equation} 
where $n,m$ are any positive integers. 
Further if $|A+B|\le K|A|$ for some sets $A,B \subseteq \Gr$, then for any $n$ one has 
\begin{equation}\label{f:Pl-R+} 
    |nB| \le K^n |A| \,.
\end{equation}
If $A\subseteq \F_p$ and $\lambda \in \F_p$, $\la \neq 0$, then we write $\lambda \cdot A := \{\lambda a ~:~ a\in A\}$.
We  use representation function notations like  $r_{A+B} (x)$ or $r_{A-B} (x)$ and so on, which counts the number of ways $x \in \F_p$ can be expressed as a sum $a+b$ or  $a-b$ with $a\in A$, $b\in B$, respectively. 
For example, $|A| = r_{A-A} (0)$.
For any two sets $A,B \subseteq \F_p$ the {\it additive energy} of $A$ and $B$ is defined by
$$
 \E^{+} (A,B) = |\{ (a_1,a_2,b_1,b_2) \in A\times A \times B \times B ~:~ a_1 - b^{}_1 = a_2 - b^{}_2 \}| = \sum_\la r^2_{A-B} (\la) \,.
$$
If $A=B$, then  we simply write $\E^{+} (A)$ for $\E^{+} (A,A)$.
In a similar way, define the {\it  multiplicative energy} $\E^{\times} (A,B)$ of sets $A, B \subseteq \F_p$. 
Clearly, $\E^{+}(A,B) = \E^{+} (B,A)$ and by the Cauchy--Schwarz inequality 
\begin{equation}\label{f:energy_CS}
\E^{+} (A,B) |A \pm B| \ge |A|^2 |B|^2 \,.
\end{equation}
Given a set $A \subseteq \F_p$ and elements $s_1,\dots, s_k\in \F_p$, we write 
$$
    A^{+}_{s_1,\dots,s_k} = A^{+}_{\{s_1,\dots,s_k\}} := A \cap (A-s_1) \cap \dots \cap (A-s_k) \,,
$$
    and similar for $A^{\times}_{s_1,\dots,s_k}$, where $s_1,\dots,s_k \neq 0$, namely,
$$
 A^{\times}_{s_1,\dots,s_k} = A^{\times}_{\{s_1,\dots,s_k\}} := A \cap (As^{-1}_1) \cap \dots \cap (As^{-1}_k) \,.
$$
For any set $A$ the following holds 
\begin{equation}\label{f:inclusion_As}
    A^{+}_{s_1,\dots,s_k} + \{0,s_1,\dots,s_k\} \subseteq A \,,
    \quad \quad
    \mbox{and}
    \quad \quad
    A^{\times}_{s_1,\dots,s_k} \cdot \{1,s_1,\dots,s_k\} \subseteq A \,.
\end{equation}

If $P_1,\dots, P_d \subset \F_p$ are arithmetic progressions, then $Q:=P_1+\dots+P_d$ is a {\it generalized arithmetic progression} (GAP) of rank $\mathrm{rk}(Q) = d$. A generalized arithmetic progression, $Q,$ is called to be {\it proper} if $|Q| = \prod_{j=1}^d |P_j|$. For properties of generalized arithmetic progressions consult, e.g., \cite{TV}. 

For any function $f:\F_p \to \mathbb{C}$ and $\rho \in \F_p$ define 
the Fourier transform of $f$ at $\rho$ by the formula 
\begin{equation}\label{f:Fourier_representations}
\FF{f} (\rho) = \sum_{g\in \F_p} f(g) e(-g\rho) \,,
\end{equation}
where $e(x):=  e^{2\pi i x/p}$.  We use the same capital letter to denote  set $A\subseteq \F_p$ and   its characteristic function $A: \F_p \to \{0,1 \}$.

Let us generalize the notions of $S(h)$ and $d(p)$. In this paper we have deal with the set 
of quadratic residues 
but our arguments equally work for the set $\mathcal{N}$.
More concretely, 
given a set $A \subset \F_p$ and a family of sets $\mathcal{F}$ write
\[
    d_\mathcal{F} (p) = \max \{ |A| ~:~ A \subseteq \mathcal{R},\, A\in \mathcal{F} \} \,,
\]
    and 
\[
    S(A) = \max \{ |S| ~:~ S\subset \F_p,\, S\dotplus A \subseteq \mathcal{R} \} \,.
\]

The signs $\ll$ and $\gg$ are the usual Vinogradov symbols. 
When the constants in the signs  depend on a parameter $M$, we write $\ll_M$ and $\gg_M$.
Let us denote by $[n]$ the set $\{1,2,\dots, n\}$.
All logarithms are to base $2$.
For a prime number $p$ we write $\F_p = \Z/p\Z$ and $\F^*_p = \F_p \setminus \{0\}$. 

\section{The proof of the main result}
\label{sec:proof}

We need an auxiliary  result in the spirit of paper \cite{shao2015character}. 
Our proof, basically, 
repeats the argument of \cite[Lemma 2]{BKS_alg}.

\begin{lemma}
    Let $h_* \le h$ be positive integers, $8h_* h<p$ and $Q$ be the set of the primes in $[h_*/2,h_*]$. Also, let $A:=[h]\dotplus S$.
    Then 
\begin{equation}\label{f:E(Q,A)}   
    \E^\times (Q,A) \ll \frac{|Q|^2 |A|^2}{p} + |Q| |A| \cdot \min \left\{ \frac{h_*}{\log h_*}, \frac{\log p}{\log h_*} \right\} \,.
\end{equation}
\label{l:E(Q,A)}
\end{lemma}
\begin{proof}
    Writing $\ov{H} =[-h,h]$ and $\ov{A}=S+\ov{H}$, we see that
\[
    \sigma:= \E^\times (Q,A) \ll h^{-2} \sum_{\la} r^2_{Q^{-1}(\ov{H}+\ov{A})} (\la) 
\]
\[
    = h^{-2}
    \sum_{\la} \sum_{q,q'\in Q} |\{ (x,x') \in \ov{H} \times \ov{H} ~:~ q'x-qx' \equiv \la \}| \cdot |\{ (a,a') \in \ov{A} \times \ov{A} ~:~ q a' - q'a \equiv \la \}| \,. 
\]
    Below we can consider the case $q\neq q'$ only because the contribution of $q=q'$ corresponds to the second term in \eqref{f:E(Q,A)}. 
    Fix $\la$ and suppose that equation $q'x-qx' \equiv \la$ has one more solution $(x_*,x'_*)$. Then $q'(x-x_*) \equiv q(x'-x'_*) \pmod p$ and since $4 h_* h <p$, we see that the later equation is, actually, an equation in $\Z$. Hence $x \equiv x_* \pmod q$ and  $x' \equiv x'_* \pmod {q'}$. 
    Therefore we have at most $O(h/h_*)$ solutions.
    Thus 
\[
     h_* h \sigma \ll \sum_{|\la|\le 2h_* h} r_{Q\ov{A}-Q\ov{A}} (\la) \,.
\]
    Using the Erd\H{o}s--Tur\'an inequality for the sequence of fractional parts $\{ \frac{qa'-q'a}{p} \}$, where $(a,a',q,q')\in \ov{A} \times \ov{A} \times Q \times Q$ (see \cite[Lemma 1]{BKS_alg}) and writing $T=[\lceil p/(8h h_*) \rceil]$, we obtain by the Cauchy--Schwarz inequality 
\[
    \sigma \ll \frac{|Q|^2 |\ov{A}|^2}{p} + \frac{1}{p} \sum_{t\in T} \left| \sum_{q\in Q} \sum_{a\in \ov{A}} e(tqa) \right|^2 
    \ll 
    \frac{|Q|^2 |A|^2}{p} + \frac{|Q|}{p} \sum_{t\in T} \sum_{q\in Q} |\FF{\ov{A}} (tq)|^2
\]
\[
    = 
     \frac{|Q|^2 |A|^2}{p} + \frac{|Q|}{p} \sum_\la |\FF{\ov{A}} (\la)|^2 r_{QT} (\la) \,.
\]
    Let us estimate the function $r_{QT} (\la)$. 
    If $\la \equiv qt \equiv q't' \pmod p$, then $\la = qt = q't'$ 
    thanks to $p/(2h_* h) \cdot h_* <p$. 
    It follows that the quantity $r_{QT} (\la)$ does not exceed 
    the number of possible primes $q\in [h_*/2,h_*]$ which divide $\la$. 
    In other words, 
    $r_{QT} (\la) = O(\min \{ h_*/\log h_*, \log_{h_*} p\})$. 
    Applying the Parseval identity, we derive 
\[
    \sigma \ll \frac{|Q|^2 |A|^2}{p} + |Q| |A| \cdot \min \left\{ \frac{h_*}{\log h_*}, \frac{\log p}{\log h_*} \right\} \,.
\]
This completes the proof.
$\hfill\Box$
\end{proof}

\bp 

Now we are ready to prove Theorem \ref{t:distribution}. 

\bigskip 

\begin{proof} First of all, let us recall how to obtain the bound $d(p) \ll p^{1/4} \log p$.  To do this we repeat the argument of \cite[Appendix]{sh_BG} and our main bound \eqref{f:distribution} will be obtained similarly.

Let $l$ be an integer parameter which we will choose later. 
Write  $P_s = s + \{0,1,\dots,l-1\} = s + P_0$ for an arbitrary arithmetic progression with step one. 
For any positive integer $m\le l-1$ we use the notation $P^{1/m}_0 = \{0,1,\dots,[(l-1)/m]\}$ and, similarly, $P^{1/m}_s = P^{1/m}_0 + s$. 
Also, let $L=\log p$. 
At the beginning choose 
$l=d(p)$ and take $P_s \subseteq \mathcal{N}$. 
We have
\begin{equation}\label{f:inclusion_R1}
    \frac{P^{1/2}_0}{P^{1/2}_s} \subseteq \mathcal{R}-1 
\end{equation}
due to 
\[
    \frac{P^{1/2}_0}{P^{1/2}_s} + 1 \subseteq \frac{P^{1/2}_0+P^{1/2}_s }{P^{1/2}_s} \subseteq \frac{P^{1}_s}{P^{1/2}_s} \subseteq \frac{\mathcal{N}}{\mathcal{N}} =  \mathcal{R} \,.
\]
Let $k$ be an integer parameter which we will choose in a moment. 
Since for any $j\in [k]$ one has  $j \cdot P^{1/2k}_0 \subseteq P^{1/2}_{0}$,  it follows from \eqref{f:inclusion_R1} that 
\begin{equation}\label{f:main_inclusion}
    \frac{P^{1/2k}_0}{P^{1/2}_s} \subseteq (\mathcal{R}-1) \cap (2^{-1}\cdot \mathcal{R}-2^{-1}) \cap \dots \cap (k^{-1} \cdot  \mathcal{R}-k^{-1}) \,.
\end{equation}
    One can show (see \cite[Lemma 26]{sh_BG})
    that $\left|\frac{P^{1/2k}_0}{P^{1/2}_s} \right| \gg |P^{1/2k}_0| |P^{1/2}_s|$
    and using formula \eqref{f:R_s} (for a mixed combination of residues/non--residues), we get
\begin{equation}\label{tmp:10.09_1}
    d^2 (p) k^{-1} \ll |P^{1/2k}_0| |P^{1/2}_a| \le 
    \frac{p}{2^{k}} + k\sqrt{p}  \ll k\sqrt{p} \,,
\end{equation}
    provided $k\gg L$. 
    Choosing $k\sim L$, 
    we 
    obtain  
    $d(p) \ll p^{1/4} \log p$ as required.

    Now we are ready to obtain \eqref{f:distribution} and we consider the quantity $S^{\mathcal{R}} (h)$ only because for $S^{\mathcal{N}} (h)$ the argument is the same. 
    Choose a set $S$ such that $|S|=S(h)$ and $A:= S\dotplus \{ 0,1,\dots,h-1\} \subseteq \mathcal{R}$. 
    Take $k=2^{-1} \min\{h/2, \log p\} - 1$. 
    From \eqref{cond:h} it follows that 
    $\frac{1}{2} \log p - \log (C\log p) \le k \le \frac{1}{2} \log p-1$ and $l:=h \ge 4k+1$.
    As above we see that our main inclusion \eqref{f:main_inclusion} takes place for any $s \in S$ and hence arguing as in \eqref{tmp:10.09_1}, we have 
\begin{equation}\label{tmp:10.09_2}
   \left| \frac{P^{1/2k}_0}{P^{1/2}_0 \dotplus S} \right| \ll C_* L \sqrt{p} \,.
\end{equation}
    Thus it is enough to obtain a good lower bound for the ratio set from \eqref{tmp:10.09_2} and to do this we calculate the multiplicative energy of a subset of this set.
    Take the set of primes $Q$ belonging the second part of $P^{1/2k}_0$ and 
    apply 
    Lemma \ref{l:E(Q,A)} to the sets $Q$ and $P^{1/2}_0 \dotplus S$ with the parameters $h_*= h/2k$ and $h=h/2$ (it is easy to check that the first term in \eqref{f:E(Q,A)}  is negligible and that the condition of the lemma takes place, since $h\le d(p) \ll p^{1/4}L$).
    It remains to combine this result with the Cauchy--Schwarz inequality \eqref{f:energy_CS} and estimate \eqref{tmp:10.09_2}.
    It gives us 
\[
   \frac{|S| h^{2}}{k \log^2 \left(\frac{4h}{\min^{} \{h/2, \log p\}} \right)\cdot \d_p (h)} 
   \ll   \frac{|Q|^2 |S|^2 h^2}{h |S| |Q|} \cdot \left(\min \left\{ \frac{h_*}{\log h_*}, \frac{\log p}{\log h_*} \right\} \right)^{-1}
\]
\begin{equation}\label{f:bound_via_energy}
   \ll
   \frac{|Q|^2 |S|^2 h^2}{\E^\times (Q,A)}
   \ll C_* L\sqrt{p} 
   \,,
\end{equation}   
    where 
$$
    \d_p (h) = \min \left\{ \frac{h}{\min \{ h/2, \log p\} },  \log p \right\} 
$$    
    as required. 

    It remains to obtain the same bound for the quantities $S^{\mathcal{Q}_{a,b}} (h)$ and $S^{\mathcal{Q}'_{a,b}} (h)$.
    Let us start with the first one. 
    Shifting (that is writing $x\to x(b-a)-a = x-a$) we can assume that $a=0$ and $b=1$ and we see that if $x,x+1,\dots,x+h-1\in \mathcal{Q}_{0,1}$, then $\frac{x+1}{x},  \frac{x+2}{x+1}, \dots, \frac{x+h}{x+h-1} \in \mathcal{R}$. 
    Multiplying we obtain that all nonzero numbers among $x,x+1,\dots,x+h$ either belong to $\mathcal{R}$ or $\mathcal{N}$. In other words, we have found a gap in $\mathcal{R}$ or $\mathcal{N}$ and hence we can apply the previous arguments. 
    As for $S^{\mathcal{Q}'_{a,a+1}} (h)$ we see that in this case a half of $x,x+1,\dots,x+h$ belongs to $\mathcal{N}$ or $\mathcal{R}$.  
This completes the proof.
%
%
$\hfill\Box$
\end{proof}

\bp 

Since 
$xy\in \mathcal{R}$ iff $x/y \in \mathcal{R}$ for any nonzero $y$, we obtain 

\begin{corollary}
    For any $a\in \F_p$  one has 
\[
    \max_{h,s} |\{ h ~:~ (x+a)(x+a+1) \in \mathcal{R}\,, x \in [h] +s  \}|
    \ll p^{1/4} \log p \,.
\]
\label{c:quadratic}
\end{corollary}

The essence of Corollary \ref{c:quadratic} is uniformity in $a$ and $s$. Without this 
uniformity 
much more general results are known, see \cite{Burgess_eps}.

Now let us formulate a simple consequence of the argument above to sums of multiplicative characters with some quadratic polynomials. 
By $d_*(p)$ denote the maximum of the maximal distance in $\mathcal{R}$ and $\mathcal{N}$.  
Also, let $\chi$ be the Legendre symbol.

\begin{corollary}
    For any $a,s\in \F_p$ and a positive integer $h$, we have 
\begin{equation}\label{f:quadratic_exp}
    \left|\sum_{x=s+1}^{s+h} \chi((x+a)(x+a+1)) \right|
    \le h (1- 2^{-1} d^{-1}_*(p)) + 4 \,.
\end{equation}  
\label{c:quadratic_exp}
\end{corollary}
\begin{proof} 
Without loss of generality we can  assume that $s=a=0$. 
    First of all, let us obtain the upper bound for the sum $ \sum_{x\in [h]} \chi(x(x+1))$. 
    We have 
    \begin{equation}\label{f:1st_splitting}
        [h] = \left(\bigsqcup_{j=1}^{n_R} R_j\right) \bigsqcup \left(\bigsqcup_{j=1}^{n_N} N_j \right) \,,
    \end{equation}
    where $R_j \subseteq \mathcal{R}$ and $N_j \subseteq \mathcal{N}$ are intervals.
    Clearly, $|n_R-n_N|\le 1$ and 
\[
    \sum_{x\in [h]} \chi(x(x+1)) = h - 2(n_R+n_N) + 1 \,.
\]
    We have $|R_j|, |N_j| \le d_* (p)$ and thus from \eqref{f:1st_splitting}, we obtain $n_R+n_N \ge h/d_*(p)$ as required. As for the lower bound, we need a  splitting other than \eqref{f:1st_splitting}.  
    Write 
 \begin{equation}\label{f:2nd_splitting}
        [h] = \left(\bigsqcup_{j=1}^{n'_R} R_j\right) \bigsqcup \left(\bigsqcup_{j=1}^{n'_N} N_j \right) \bigsqcup \left(\bigsqcup_{j=1}^{s} I_j \right) \,,
    \end{equation}
    where  $R_j \subseteq \mathcal{R}$ and $N_j \subseteq \mathcal{N}$ are intervals of length at least two, and intervals $I_j$ of quadratic residues/non--residues form the complement of $R_j,N_j$ in $[h]$.
    Again $|n'_R+n'_N - s|\le 1$ and   now we have in view of $|R_j|, |N_j|\ge 2$ that 
\[
    \sum_{x\in [h]} \chi(x(x+1)) = \sum_{j=1}^{n'_R} |R_j| + \sum_{j=1}^{n'_N} |N_j| - \sum_{j=1}^s |I_j| - n'_R - n'_N - s
\]
\[
    =  -h + 2\sum_{j=1}^{n'_R} |R_j| + 2\sum_{j=1}^{n'_N} |N_j| - n'_R - n'_N - s 
    \ge -h + 3n'_R + 3n'_N - s \ge - h  +2s -3 \,.
\]
    It is easy to see that  $|I_j| \le 2d_* (p)$ (alternatively, we can use the quantity $2d(p)+1$ here) and thus 
\[
   2s+1\ge s + n'_R + n'_N \ge h/(2d_*(p)) \,.
\]
    Combining the last two inequalities, we obtain the required result.
$\hfill\Box$
\end{proof}

\begin{remark}
    Thanks 
    to the fact that $\xi \cdot \mathcal{R}$ is equal to either $\mathcal{R}$ or $\mathcal{N}$ for any nonzero $\xi$ 
    one can deal with the general set $\mathcal{Q}_{a,b}$. 
    The only difference is that then we need to consider arithmetic progressions with step $b-a$ in this case.
\end{remark}

Using Theorem \ref{t:distribution} we can break the square--root barrier for sufficiently large arithmetic progressions $[h]$.  

\begin{corollary}
    Let $h$ be a number satisfying condition \eqref{cond:h}.
    Then
\begin{equation}\label{f:R_h}
    |\mathcal{R}^{+}_{[h]}| \ll \frac{C_* \sqrt{p}\cdot \log^2 p  \cdot \log^2 \left(\frac{4h}{\min^{} \{h, 2\log p\}}\right) \cdot \d_p (h)}{h} \,.
\end{equation}
\label{c:R_h}
\end{corollary}
\begin{proof}
    We use the Ruzsa covering argument.
    By \eqref{f:inclusion_As} we have $\mathcal{R}^{+}_{[h]}+[h] \subseteq \mathcal{R}$. 
    Take the maximal  set 
    $X \subseteq \mathcal{R}^{+}_{[h]}$
    such that  the sum $X +[h]$ is direct. Then it is easy to see that 
\[
    \mathcal{R}^{+}_{[h]}  \subseteq [h]-[h] + X 
\]
    and hence $|X| \ge |\mathcal{R}^{+}_{[h]}|/(2h)$. 
    Applying Theorem \ref{t:distribution}, we obtain 
\[
    h |\mathcal{R}^{+}_{[h]}| \ll h^2 |X| \ll C_* \sqrt{p}\cdot \log^2 p \cdot \log^2 \left(\frac{4h}{\min^{} \{h, 2\log p\}}\right) \cdot \d_p (h)
\]
    as required. 
$\hfill\Box$
\end{proof}

\begin{remark}
\label{r:optimality}
    Let us show that condition \eqref{cond:h} and bound \eqref{f:distribution} are close to optimal. 
    Indeed, suppose that condition \eqref{cond:h} can be weakened by about a factor of 
    four. 
    Namely, choosing  $h = \frac{1}{2} \log p - \log (C\log p)$ for sufficiently large $C$, we obtain from \eqref{f:R_s} (actually, we need  the asymptotic formula here) that $|\mathcal{R}^{+}_{[h]}|\gg \sqrt{p} \cdot \log p$ and this lower bound coincides with upper estimate \eqref{f:R_h}. 
\end{remark}

Now we consider a more general question (see, e.g., \cite{alsetri2022hilbert}) about maximal GAP in the set of quadratic residues. 

\begin{corollary}
    Let $G = \{ \a_1 x_1 + \dots + \a_r x_r + \beta ~:~ x_j \in [h] \}$ be a GAP such that $\{ \a_1 x_1 + \dots + \a_r x_r + \beta ~:~ |x_j| \le h^2 \}$ is a proper GAP. 
    Suppose that $G\subseteq \mathcal{R}$.
    Then  
\begin{equation}\label{f:Kerr1}
    |G| \ll p^{1/4} (2\log p)^{3r/2+1} \,.
\end{equation}
    In particular, for any $\eps \in (0,1/2)$ and any proper GAP $G$ such that $G\subseteq \mathcal{R}$ and 
    $$
    \mathrm{rk} (G) \ll \frac{\eps^3  \log p}{\log \log p}
    $$ 
    one has 
\begin{equation}\label{f:Kerr2}
    |G| \ll p^{1/4+\eps} \,.
\end{equation}
\label{c:Kerr}
\end{corollary}
\begin{proof}
    Let $L = \log p$.
    Also, we use the notation  $G=G(h)$, $G' = G(h/2)$ and let 
    $$G_0 = G(h/2k) - \beta = \{ \a_1 x_1 + \dots + \a_r x_r~:~ x_j \in [h/2k] \} \,,
    $$
    where $k=L$. 
    Then we have an analogue of inclusion \eqref{f:main_inclusion}, namely, 
\begin{equation*}
    \frac{G_0}{G'} \subseteq (\mathcal{R}-1) \cap (2^{-1}\cdot \mathcal{R}-2^{-1}) \cap \dots \cap (k^{-1} \cdot \mathcal{R}-k^{-1}) \,.
\end{equation*}
    It remains to apply \cite[Corollary 4]{Kerr_GAP_energy},  which says that $\E^\times (\tilde{G}) \ll |\tilde{G}|^2 L^{2r+1}$ for any GAP $\tilde{G}$ such that $\tilde{G}(h^2)$ is a proper GAP and use the Cauchy--Schwartz inequality \eqref{f:energy_CS}. 
    It gives us
\[
   |G|^2 L^{-(2r+1)}  (4L)^{-r} \ll |G_0| |G'| L^{-(2r+1)} \ll  L\sqrt{p}
\]
    as required. 

    To 
    get 
    \eqref{f:Kerr2}  we need to use some arguments of \cite{Kerr_GAP_energy} again. 
    Let $G = \{ \a_1 x_1 + \dots + \a_r x_r + \beta ~:~ x_j \in [h_j] \}$ and remove all $h_j$ such that $h_j \le |G|^{\eps/(16r)}$.
    We obtain a GAP $G'\subseteq G$, $|G'| \ge |G|^{1-\eps/16}$ with rank $r'=\mathrm{rk} (G') \le r$. Without loss of generality let $G' = \{ \a_1 x_1 + \dots + \a_{r'} x_{r'} + \beta ~:~ x_j \in [h_j] \}$ and let $h=\min_{j\in [r']} h_j \ge |G|^{\eps/(16r)}$. 
    Write any $x_j$, $j\in [r']$ in base $[h^{\eps/16}]$. 
    Then we obtain a GAP $G_* \subseteq G$, $|G_*| \ge |G|^{1-\eps/4}$ of rank 
\[
    \mathrm{rk} (G_*) \ll \sum_{j=1}^{r'} \frac{\log h_j}{\log (h^{\eps/16})}
    \ll
    r \eps^{-2} (\log |G|)^{-1} \sum_{j=1}^{r'} \log h_j
    \le 
    r \eps^{-2} 
\]
    and it is easy to see that we can apply the previous argument to $G_*$.
    It gives us 
\[
    |G|^{1-\eps/4} \le |G_*| \ll p^{1/4} L^{O(\eps^{-2} r^{})}
\]
    and hence choosing $r\ll \eps^3 L/\log L$, we obtain 
\[
    |G|\ll p^{1/4 + \eps/2 + \eps/2} = p^{1/4+\eps} 
\]
    as required. 
$\hfill\Box$
\end{proof}

\bp 

If one combine the energy bound of \cite{Kerr_GAP_energy} and the Burgess method (see, e.g., \cite{chang2008question}), then it is possible  to obtain a non--trivial upper bound for the sums of multiplicative characters over GAP $G$ of size $|G|\gg p^{1/4+\eps}$ (this improves 
\cite[Proposition 3.2]{alsetri2022hilbert}). We prefer to have a simple argument to get \eqref{f:Kerr2}.
Also, let us remark that our method  of 
obtaining 
upper bounds for $S(h)$ (see, e.g., Corollary \ref{c:Kerr}) is rather general, and if we have a non--trivial lower bound for $|S\dotplus \{ 0,1,\dots,h-1\}|$, see definition \eqref{def:S(h)}, then we obtain a refinement of Theorem \ref{t:distribution}. 
Also, it is easy to see that one can replace the segment $\{0,1,\dots,h-1\}$ in formula \eqref{def:S(h)} with  its arbitrary subset of density 
$1-O\left( \frac{1}{\log^{} h \cdot \log^{2} p} \right)$ 
and obtain some non--trivial results. 


\section{Additive--combinatorial considerations}
\label{sec:AC}

Recall that 
given a set $A \subset \F_p$ and a family of sets $\mathcal{F}$ we write
\[
    d_\mathcal{F} (p) = \max \{ |A| ~:~ A \subseteq \mathcal{R},\, A\in \mathcal{F} \} \,,
\]
    and 
\[
    S(A) = \max \{ |S| ~:~ S\subset \F_p,\, S\dotplus A \subseteq \mathcal{R} \} \,.
\]
Thus $S(h)=S(\{0,1,\dots,h-1\})$ and $d(p) =  d_{\mathcal{F}_{AP}} (p)$, where $\mathcal{F}_{AP}$ is the family of all arithmetic progressions in $\F_p$ with step one. 
In Corollary \ref{c:Kerr} we have estimated the quantity $d_{\mathcal{F}_{GAP}} (p)$ for the family of all GAP's in $\F_p$. 
Of course, the argument of the previous section works for sets of the form $A=XY+Z$, where $X,Y,Z\subset \F_p$, exactly as in the Burgess method, but we show that it is possible to estimate the quantities  $d_{\mathcal{F}_{SD}} (p)$, $S(A)$ for the family of sumsets of sets with small doubling.

\bp 

We need a lemma 
on the 
multiplicative structure of additively rich sets.
Of course, 
a similar result 
holds for rather general rings, but in this section we 
restrict 
ourselves to the case of $\F_p$.

\begin{lemma}
	Let $A\subseteq \F_p$ be a set such that  $|A+A| \le K|A|$.
	There exists an absolute constant $C >0$
	 such that for any positive integers  $d\ge 2$ and $l$ there is an element $a_*\in A$  and 
	a set $Z \subseteq A-a_*$ of size  $|Z| \ge   \exp \left(-C l^{3} d^2 (\log K)^2 \right) |A|$ with
\begin{equation}\label{f:mult_inclusion}	
[d^l] \cdot Z \subseteq 2A-2A \,.
\end{equation}
    Similarly,  there exists an element $a_*\in A$ and a set $W\subset A-a_*$ of size  $|W| \ge |A|/K^{3d-3}$ with
\begin{equation}\label{f:mult_inclusion_new}	
    [d] \cdot W \subseteq A-A \,.
\end{equation}
\label{l:mult_inclusion}
\end{lemma}
\begin{proof}
    Inclusion \eqref{f:mult_inclusion} coincides with \cite[Proposition 10]{Schoen_sh_Balog} and the second formula \eqref{f:mult_inclusion_new} is
    actually 
    contained in this paper. 
    Indeed, consider the map $f:A^d \to \F^{d-1}_p$, 
\[
    f(\vec{a}) = f(a_1,a_2,\dots,a_d) = (a_2-2 a_1, \dots, a_d-d a_1) := (x_1,\dots,x_d)\,. 
\]
    It is easy to see that $x_1 \in A-2A$ and 
\[
    x_{j+1} - x_j = a_{j+1} - a_j - a_1 \in A-2A \,, \quad \quad j=1,\dots,d-1 \,.
\]
    Hence there is a vector $\vec{v}\in \F^{d-1}_p$ such that 
\[
    |\{ \vec{a} \in A^d ~:~ f(\vec{a}) = \vec{v} \}| \ge \frac{|A|^d}{|A-2A|^{d-1}} \ge \frac{|A|}{K^{3(d-1)}}
\]
    thanks to the Pl\"unnecke inequality \eqref{f:Pl-R}. 
    Thus for a fixed vector $(b_1,b_2,\dots,b_d)$ there are at least $\frac{|A|}{K^{3(d-1)}}$ solutions to the equation 
\begin{equation}\label{tmp:18.08_1}
    j (a_1-b_1) = a_j - b_j \in A-A \,, \quad \quad j=2,\dots,d \,.
\end{equation}   
    Let $W$ be the  set of elements $a_1$ satisfying \eqref{tmp:18.08_1} and put $a_*:=b_1$. 
    Then $[d] \cdot W \subseteq A-A$. 
This completes the proof.
$\hfill\Box$
\end{proof}

\bp 

We need \cite[Theorem 29]{collinear}. 

\begin{theorem}
    Let $A\subseteq \F_p$, $|A+A|\le K|A|$ and $|A| \le p^{13/23} K^{25/92}$. 
    Then 
\[
    \E^{\times} (A) \ll K^{51/26} |A|^{32/13} \log^{18/13} |A|
\]
\label{t:32/13_prod}
\end{theorem}


Also, we need the Rudnev points/planes incidences result 
\cite{Rudnev_points/planes}. 

\begin{theorem}
	Let $p$ be an odd prime, $\mathcal{P} \subseteq \F_p^3$ be a set of points and $\Pi$ be a collection of planes in $\F_p^3$. 
	Suppose that $|\mathcal{P}| \le |\Pi|$, $|\mathcal{P}| \le  p^2$ and that $k$ is the maximum number of collinear points in $\mathcal{P}$. 
	Then 
	 \begin{equation}\label{f:Misha+_a}
	 |\{ (q,\pi) \in \mathcal{P} \times \Pi ~:~ q\in \pi \}|- \frac{|\mathcal{P}| |\Pi|}{p} \ll |\mathcal{P}|^{1/2} |\Pi| + k |\mathcal{P}| \,.	
	 \end{equation}
	\label{t:Misha+}
\end{theorem}

Now we are ready 
to prove the main result of this section, which generalizes the obtained upper bound for $S(h)$ and improves some results of \cite{Schoen_sh_Balog} and \cite{volostnov2018double}  applicable to  larger sets of size $p^{1/3+\eps}> p^{13/40}$.  
Our estimates \eqref{f:D_R_1}, \eqref{f:D_R_2} below work for subexponentially small doubling constants. 

\begin{theorem}
    Let $A\subseteq \F_p$, $|A+A|\le K|A|$. 
    Suppose that $3A-2A \subseteq \mathcal{R}$.
    Then 
\begin{equation}\label{f:D_R_1}
    |A| \ll p^{13/40} \exp (O(  (\log K)^2 (\log \log p)^3))\,.
\end{equation} 
    Also, for any set $A$ one has 
\begin{equation}\label{f:D_R_2}
    S(4A-4A) \ll \frac{p}{|A|^3} \cdot \exp \left(O( l^{3} (\log K)^2) \right) \,.
\end{equation} 
\label{t:D_R}
\end{theorem}
\begin{proof}
    Let $L=\log p$, $D=A-A$ and $T=D-D$. 
    As in the proof of Theorem \ref{t:distribution} we have $\frac{T}{A} \subseteq \mathcal{R}-1$ (the set $A$ automatically belongs to $\mathcal{R}$).
    Applying the first part of Lemma \ref{l:mult_inclusion} with $d=2$ and $l\sim \log L$, we find a set $Z$, $|Z| \ge   \exp \left(-C l^{3} (\log K)^2 \right) |A|$ such that $\frac{Z}{A} \subseteq \tilde{\mathcal{R}}^{+}_{[2^l]^{-1}}$.
    Here and below $C>0$ is an absolute constant that  can change from line to line and we write $\tilde{\mathcal{R}}^{+}_{[2^l]^{-1}}$ for the intersection of these $2^l$ shifts of $\mathcal{R}$ or $\mathcal{N}$.
    Using the Cauchy--Schwarz inequality \eqref{f:energy_CS} and bound \eqref{f:R_s}, one gets 
\[
    \frac{|Z|^2 |A|^2}{\sqrt{\E^{\times} (A) \E^{\times} (Z)}} \le \left| \frac{Z}{A} \right| \ll L \sqrt{p} \,.
\]
    It remains to recall that $Z\subseteq A-a_*$ for a certain $a_*$ and  apply Theorem \ref{t:32/13_prod}. 
    We obtain 
\[
    |A|^4 \ll K^{51/26} |A|^{32/13} L^{31/13} \sqrt{p} \cdot \exp (C l^3 (\log K)^2) \ll |A|^{32/13} \sqrt{p} \cdot \exp (C l^3 (\log K)^2) 
\]
    and hence 
\[
    |A| \ll p^{13/40} \exp (C l^3 (\log K)^2) 
\]
    as required. It is easy to check that the condition  $|A| \le p^{13/23} K^{25/92}$ takes place, e.g., use previous known bounds \cite{Schoen_sh_Balog} or \cite{volostnov2018double} to estimate  size of $A$ or just apply the Gauss sums.

    Now let us obtain \eqref{f:D_R_2}.
    Let $S$ be the maximal set such that $S\dotplus 2T \subseteq \mathcal{R}$. 
    Since $D \subseteq T\subseteq 2T$, it follows that $\frac{T}{S\dotplus T} \subseteq \mathcal{R}-1$ and applying the first part of  Lemma \ref{l:mult_inclusion}  as above, we obtain 
$$
   \frac{Z}{(S \dotplus D) +D} = \frac{Z}{S\dotplus T} \subseteq  \tilde{\mathcal{R}}^{+}_{[2^l]^{-1}} \,.
$$
    Using 
    Theorem \ref{t:Misha+} 
    (or just see \cite[Lemma 22]{collinear}), we obtain  
\[
    |D| \sqrt{|A| |S|} \cdot \exp \left(-C l^{3} (\log K)^2 \right)  \ll \sqrt{|Z| |D| |S \dotplus D|} \ll \left| \frac{Z}{(S \dotplus D) +D} \right| \ll L \sqrt{p} \,.
\]
    It is easy to see that $|S||Z||D| \le (|S| |A|) |D| \le p|D| \le p^2$ and thus the Rudnev theorem can be applied. 
    Finally, we get 
\[
    |S| \ll \frac{p}{|A|^3} \cdot \exp \left(C l^{3} (\log K)^2 \right) \,.
\]
This completes the proof.
$\hfill\Box$
\end{proof}

\bp

Now we say a few words about general combinatorial properties of sets $A^{+}_{s_1,\dots,s_l}$ (and $A^{\times}_{s_1,\dots,s_l}$). The results below reflect the hypothesis that such sets must have a rich additive structure.  
For example, bound \eqref{f:As_cap_As*} 
says, in particular, that if size of $A^{+}_S$ is comparable with $|A|$, then the multiplicative energy of $A^{+}_S$ must be small. Also, it is known \cite{CS} that sets $A^{+}_S$ correlate with sets of almost periods of convolutions 
which
definitely additively rich. 
We 
enunciate 
the following 

\bp 

{\bf Question.} {\it Let $A \subset \F_p$ be a set, $|A|<\sqrt{p}$, say. Is it true that for any $S\subset \F_p$ one has 
$$|A^{+}_S \cdot A^{+}_S|\gg |A^{+}_S| \cdot |S|^c$$ for an absolute $c>0$?}

\bp

Let us formulate a positive result in the direction of the above hypothesis. 

\begin{theorem}
    Let $A\subseteq \F_p$ be a set and $S\subset \F^*_p$ be a set of shifts, $|S|\ge p^\d$, where $\d\in (0,1]$.
    Then there is $\eps(\d)>0$ such that for any $\la\neq 0$ the following holds
\begin{equation}\label{f:As_cap_As*}
    |A^{+}_S \cap \la (A^{+}_S)^{-1}| \ll |A| p^{-\eps} \,,
\end{equation}
    provided $|A|\le p^{1-\eps}$. 
    If $h$ be a positive integer, then for any arithmetic progressions $H_1$, $H_2$ of size $h$ one has 
\begin{equation}\label{f:Ah_cap_Ah*}
    |A^{+}_{H_1} \cap \la (A^{+}_{H_2})^{-1}| \ll |A| h^{-c} \,,
\end{equation}
    where $c>0$ is an absolute constant, and $|A| \le p h^{-c}$.\\ 
    Finally, for any $|A|\le p |S|^{-1/2}$, we have 
\begin{equation}\label{f:A_times}
    \E^{+} (A^\times_S) 
    \ll \frac{|A|^2 |A^\times_S|^2}{p} + |A^\times_S|^{3/2} |A|^{3/2} |S|^{-1/2} 
    \ll 
    |A^\times_S|^{3/2} |A|^{3/2} |S|^{-1/2} \,.
\end{equation}
\label{t:As_cap_As*}
\end{theorem}
\begin{proof}
    In view of inclusion \eqref{f:inclusion_As} one has 
\[
    A^{+}_S (x) \le |S|^{-1} r_{A+S} (x) \,,
\]
    and thus 
\[
    |A^{+}_S \cap \la (A^{+}_S)^{-1}| \le |S|^{-2} r_{(A+S)(A+S)} (\la) \,.
\]
    It remains to use a consequence of the Bourgain--Gamburd machine see, e.g.,   \cite{sh_BG}. It follows that 
\begin{equation}\label{tmp:22.08_1}
    |A^{+}_S \cap \la (A^{+}_S)^{-1}| \le |S|^{-2} \left( \frac{|A|^2|S|^2}{p} + O(|A| |S|^2 p^{-\eps}) \right) = 
    \frac{|A|^2}{p} + O(|A| p^{-\eps}) = O(|A| p^{-\eps}) 
\end{equation}
    thanks to our condition $|A|\le p^{1-\eps}$.
    To get \eqref{f:Ah_cap_Ah*}, use the Bourgain--Gamburd machine again (see \cite[Theorem 2 or Theorem 3]{sh_BG}).
    Thanks to the specific form of our 
    shift sets 
    (they are arithmetic progressions) we have the saving $h^{-c}$ instead of $p^{-\eps}$ and therefore it is possible to obtain   \eqref{f:Ah_cap_Ah*} 
    exactly as  
    \eqref{tmp:22.08_1}. 
    It remains to derive \eqref{f:A_times}. 
    As above, we have 
\[
    A^{\times}_S (x) \le |S|^{-1} r_{AS} (x) 
\]
    thanks to inclusion \eqref{f:inclusion_As} and hence by Theorem \ref{t:Misha+}, we get
\[
    \E^{+} (A^\times_S) \le |S|^{-2} \sum_x r^2_{AS+A^\times_S} (x) 
    \ll 
    \frac{|A|^2 |A^\times_S|^2}{p} + |A^\times_S|^{3/2} |A|^{3/2} |S|^{-1/2}
\]
    as required. 
$\hfill\Box$
\end{proof}

\bp 

Let us make a general remark. 
Since 
\[
    A^{+}_{X+Y} = (A^{+}_{X})_Y = (A^{+}_{Y})_X \,,
\]
it follows that by Theorem \ref{t:As_cap_As*} 
\begin{equation}\label{f:A_X+Y}
    |A^{+}_{X+Y} \cap ((A^{+}_X)_Y)^{-1}| \ll |A^{+}_X| |Y|^{-\eps(\d)} 
\end{equation}
    for any $|Y|\ge p^\d$, $|A_X| \le p^{1-\eps(\d)}$ and similarly in the case of arithmetic progressions $X$ or $Y$.
    Using estimate \eqref{f:A_X+Y} we break the square--root barrier for quadratic residues one more time.

\begin{corollary}
    Let $X\subseteq \F_p$ be a set and $H \subseteq \F_p$ be an arithmetic progression. 
    Suppose that $|X| \ge \log p$ and $H+X \subseteq \mathcal{R}$. 
    Then 
\begin{equation}\label{f:R_R*}
    |\mathcal{R}^{+}_{(H+X) \cup (H+X)^{-1}}| \ll \frac{\sqrt{p} \log p }{|H|^c} \,,
\end{equation}
    where $c>0$ is an absolute constant. 
\label{c:R_R*}
\end{corollary}
\begin{proof}
    Indeed, by estimate \eqref{f:A_X+Y} applied  to the arithmetic progression $H$, as well as formula \eqref{f:R_s}, we get 
\[
    |\mathcal{R}^{+}_{H+X} \cap ((\mathcal{R}^{+}_X)_H)^{-1}| 
    \ll 
    |\mathcal{R}^{+}_{X}| |H|^{-c}
    \ll 
    \sqrt{p} |H|^{-c} \log p \,.
\]
    It remains to notice that $(\mathcal{R}^{+}_{H+X})^{-1} = \mathcal{R}^{+}_{(H+X)^{-1}}$, thanks to our condition  $H+X \subseteq \mathcal{R}$, see \cite[Lemma 22]{sh_BG}. 
This completes the proof.
$\hfill\Box$
\end{proof}


\bp 

Finally, let us make a remark concerning the real case. 
Let $R \subset \mathbb{R}$ be a set with small multiplicative doubling in the sense that $|RR|\le K_* |R|$ and $|RR/R|\le K|R|$ for some parameters $K_*, K>1$. 
Take any vector $\vec{s} = (0,s_1,\dots,s_{2^l-1})$ and notice that inclusion \eqref{f:inclusion_As} implies 
\[
    \prod_{j=1}^{2^l} (R_{\vec{s}} + s_j) \subseteq R^{2^l} \,.
\]
Now applying \cite[Theorem 1.4]{Higher_convexity_E} and the Pl\"unnecke inequality \eqref{f:Pl-R+}, we obtain 
\[
    |R_{\vec{s}}|^{l+1-\a_l} K^{-2^{l+1}+2+2l-2\a_l} \ll K^{2^l}_* |R|
    \,,
\]
    where $\a_l:= \sum_{j=1}^l j2^{-j}$ and 
    therefore 
\[
    |R_{\vec{s}}| \ll |R|^{\frac{1}{l+1-\a_l}} (K^{2^l}_* K^{2^{l+1}-2-2l+2\a_l} )^{\frac{1}{l+1-\a_l}} \,.
\]
    It gives us a non--trivial analogue of formula \eqref{f:R_s}, as well as an analogue of the general result on the  intersections   of shifts of  multiplicative subgroups in $\F_p$, see  \cite{sh_vyugin_subgroups}.

\section{Concluding remarks}

Let $I \subseteq \F_p$ be an interval and consider the quantity 
\[
    S^{\mathcal{R} \cap I} (h) = \max \{ |S| ~:~ S\subset \F_p,\, S\dotplus \{0,1,\dots,h-1\} \subseteq \mathcal{R} \cap I \} \,.
\]
Thus we want to study gaps of $\mathcal{R}$ or $\mathcal{N}$ in a ``local'' sense. 
It is very interesting to say something non--trivial about the quantity $S^{\mathcal{R} \cap I} (h)$, depending on the size of 
$I$. Here we formulate a simple result in this direction, which improves Theorem \ref{t:distribution} roughly by a logarithm (if the interval $I$ is  small enough).


\begin{theorem}
    Let $I \subseteq \F_p$ be an interval, $h$ be a positive integer and $16 |I|h^2 < p \log^2 p$. 
    Then in the notation of Theorem \ref{t:distribution} one has 
\begin{equation}\label{f:S_I}
    S^{\mathcal{R} \cap I} (h) \ll 
    \frac{C_* \sqrt{p} \cdot \log^2 p}{h^2} \,,
\end{equation}
\label{t:S_I}
\end{theorem}
\begin{proof}
    As in the proof of Theorem \ref{t:distribution}, we have formula 
    \eqref{tmp:10.09_2} and we need to estimate the left--hand side, that is 
    $\left|\frac{P^{1/2k}_0}{P^{1/2}_0 \dotplus S} \right|$.
    But now the sumset $P^{1/2}_0 \dotplus S$ belongs to $I$ and hence we can apply \cite[Lemma 26]{sh_BG} which says that the lower bound for $\left|\frac{P^{1/2k}_0}{P^{1/2}_0 \dotplus S} \right|$ 
     is 
$$
\Omega( |\{ (x,y)\in  P^{1/2k}_0 \times (P^{1/2}_0 \dotplus S)   ~:~ (x,y)=1 \}| )
\,.
$$
    It is easy to see (or just refer to the proof at the end of \cite[Lemma 26]{sh_BG}) that the last quantity is $\Omega(|P^{1/2k}_0| |P^{1/2}_0| |S|)$. 
This completes the proof.
$\hfill\Box$
\end{proof}

\bp

Using our combinatorial method, we can give 
rather simple 
proof of  \cite[Theorem 3]{Burgess_eps}.



\begin{theorem}
    Let $p$ be a sufficiently large prime, $\eps>0$ be an arbitrary number and $h$ be a positive integer such that $$ h \gg  p^{11/24} (\log p)^{4/3+\eta/6} (\log \log p)^{1/2} \,,$$
    where $\eta = 0.16656 < 1/6$ is a certain constant. 
    Then for any $s$ the segment $s+1,\dots, s+h$ contains a pair of consecutive quadratic residues and a  pair of consecutive quadratic non--residues. 
\label{t:Burgess_eps_0.25}
\end{theorem}
\begin{proof}
    We use the notation of the proofs of Theorem \ref{t:distribution} and Corollary \ref{c:quadratic_exp}. 
    Suppose that the segment $s+1,\dots, s+h$ does not contains a pair of consecutive quadratic residues (the argument for the non--residues is the same). Thus  the decomposition \eqref{f:2nd_splitting} has the form
\begin{equation}\label{f:2nd_splitting'}
        s+[h] = \left(\bigsqcup_{j=1}^{n'_N} N_j \right) \bigsqcup \left(\bigsqcup_{j=1}^{s} I_j \right) = \Sigma_N \bigsqcup \Sigma_I \,.
\end{equation}
    We have
\[
    \sigma :=\sum_{x\in s+[h]} \chi(x) = -|\Sigma_N| +s \,,
\]
    and therefore
\[
    2s-1\le 2n'_N \le |\Sigma_N| \le s+ \sigma \,.
\]
    In follows that $s\le \sigma +1$ and hence 
\[
    |\Sigma_N| \le 2\sigma +1 \,.
\]
    Below we definitely assume that $\sigma = o(h)$.
    Clearly, any 
    interval $I_j$ contains an arithmetic progression of length at least $|I_j|/2$ with step two belonging to $\mathcal{R}$. 
    Using the average argument and the obtained bound for $|\Sigma_N|$,  we find $S(h_*)$ disjoint arithmetic progression with step two belonging to $\mathcal{R}$, where $h_* \gg  h/\sigma$  and such that $S(h_*) h_* \gg h$. 
    Now applying  the modern form of Burgess bound \cite[Theorem 1.3]{BMT_Gal} with the parameter $r=2$, as well as Theorem \ref{t:S_I}, we get
\begin{equation}\label{tmp:h,h_*}
    h^2 \ll \sigma h_* h \ll L^2 \sqrt{p}  \sigma
    \ll L^2 \sqrt{p} \cdot h^{1/2} p^{3/16} L^{\eta/4} (\log L)^{3/4} 
\end{equation}
    and hence 
\[
    h \ll p^{11/24} L^{4/3+\eta/6} (\log L)^{1/2}
\]
as required.  It remains to check that $h^2_* h \ll p L^2$
 but it easily follows from \eqref{tmp:h,h_*}. 
This completes the proof.
$\hfill\Box$
\end{proof}



\bibliographystyle{abbrv}

\bibliography{bibliography}{}

\begin{thebibliography}{10}

\bibitem{alsetri2022hilbert}
A.~Alsetri and X.~Shao.
\newblock {On Hilbert cubes and primitive roots in finite fields}.
\newblock {\em Archiv der Mathematik}, 118:49--56, 2022.

\bibitem{BKS_alg}
J.~Bourgain, S.~Konyagin, and I.~Shparlinski.
\newblock {Character sums and deterministic polynomial root finding in finite
  fields}.
\newblock {\em Mathematics of Computation}, 84(296):2969--2977, 2015.

\bibitem{Burgess2}
D.~Burgess.
\newblock {A note on the distribution of residues and non-residues}.
\newblock {\em Journal of the London Mathematical Society}, 1(1):253--256,
  1963.

\bibitem{Burgess_eps}
D.~Burgess.
\newblock {On Dirichlet characters of polynomials}.
\newblock {\em Proceedings of the London Mathematical Society}, 3(1):537--548,
  1963.

\bibitem{Burgess3}
D.~Burgess.
\newblock {The character sum estimate with $r= 3$}.
\newblock {\em Journal of the London Mathematical Society}, 2(2):219--226,
  1986.

\bibitem{Burgess1}
D.~A. Burgess.
\newblock {On character sums and primitive roots}.
\newblock {\em Proceedings of the London Mathematical Society}, 3(1):179--192,
  1962.

\bibitem{chang2008question}
M.-C. Chang.
\newblock {On a question of Davenport and Lewis and new character sum bounds in
  finite fields}.
\newblock {\em Duke Math. J.}, 3(145):409--442, 2008.

\bibitem{CS}
E.~Croot and O.~Sisask.
\newblock {A probabilistic technique for finding almost-periods of
  convolutions}.
\newblock {\em Geometric and functional analysis}, 20:1367--1396, 2010.

\bibitem{BMT_Gal}
R.~de~la Bret{\`e}che, M.~Munsch, and G.~Tenenbaum.
\newblock {Small G{\'a}l sums and applications}.
\newblock {\em Journal of the London Mathematical Society}, 103(1):336--352,
  2021.

\bibitem{DES_digits}
R.~Dietmann, C.~Elsholtz, and I.~Shparlinski.
\newblock {Prescribing the binary digits of squarefree numbers and quadratic
  residues}.
\newblock {\em Transactions of the American Mathematical Society},
  369(12):8369--8388, 2017.

\bibitem{DES_Hamming}
R.~Dietmann, C.~Elsholtz, and I.~E. Shparlinski.
\newblock {On gaps between quadratic non-residues in the Euclidean and Hamming
  metrics}.
\newblock {\em Indagationes Mathematicae}, 24(4):930--938, 2013.

\bibitem{FI_char_sums}
J.~Friedlander and H.~Iwaniec.
\newblock {Estimates for character sums}.
\newblock {\em Proceedings of the American Mathematical Society},
  119(2):365--372, 1993.

\bibitem{Higher_convexity_E}
B.~Hanson, O.~Roche-Newton, and M.~Rudnev.
\newblock {Higher Convexity and Iterated Sum Sets}.
\newblock {\em Combinatorica}, 42(1):71--85, 2022.

\bibitem{H-B_Burgess_moments}
D.~Heath-Brown.
\newblock {Burgess's bounds for character sums}.
\newblock In {\em Number Theory and Related Fields: In Memory of Alf van der
  Poorten}, pages 199--213. Springer, 2013.

\bibitem{IK_book}
H.~Iwaniec and E.~Kowalski.
\newblock {\em {Analytic number theory}}, volume~53.
\newblock American Mathematical Soc., 2021.

\bibitem{Karatsuba_survey}
A.~A. Karatsuba.
\newblock {Arithmetic problems in the theory of Dirichlet characters}.
\newblock {\em Russian Mathematical Surveys}, 63(4):641, 2008.

\bibitem{Kerr_GAP_energy}
B.~Kerr.
\newblock {Some multiplicative equations in finite fields}.
\newblock {\em Finite Fields and Their Applications}, 75:101883, 2021.

\bibitem{collinear}
B.~Murphy, G.~Petridis, O.~Roche-Newton, M.~Rudnev, and I.~D. Shkredov.
\newblock {New results on sum-product type growth over fields}.
\newblock {\em Mathematika}, 65(3):588--642, 2019.

\bibitem{Polya}
G.~P{\'o}lya.
\newblock {{\"U}ber die Verteilung der quadratischen Reste und Nichtreste}.
\newblock {\em G{\"o}ttingen Nachrichten}, 1:21--29, 1918.

\bibitem{Rudnev_points/planes}
M.~Rudnev.
\newblock {On the number of incidences between points and planes in three
  dimensions}.
\newblock {\em Combinatorica}, 38:219--254, 2018.

\bibitem{Schoen_sh_Balog}
T.~Schoen and I.~D. Shkredov.
\newblock {Character sums estimates and an application to a problem of Balog}.
\newblock {\em Indiana Univ. Math. J.}, 3(71):953--964, 2022.

\bibitem{shao2015character}
X.~Shao.
\newblock {Character sums over unions of intervals}.
\newblock {\em Forum Mathematicum}, 27(5):3017--3026, 2015.

\bibitem{sh_BG}
I.~D. Shkredov.
\newblock {On a girth--free variant of the Bourgain--Gamburd machine}.
\newblock {\em Finite Fields and Their Applications}, 90:102225, 2023.

\bibitem{TV}
T.~Tao and V.~Vu.
\newblock {\em Additive combinatorics}, volume 105 of {\em Cambridge Studies in
  Advanced Mathematics}.
\newblock Cambridge University Press, Cambridge, 2006.

\bibitem{Vinogradov_1918}
I.~M. Vinogradov.
\newblock Sur la distribution des r{\'e}sidus et des non-r{\'e}sidus des
  puissances.
\newblock {\em J. Phys.-Math. Soc. Perm}, 1(1):94--98, 1918.

\bibitem{Vinogradov_selected}
I.~M. Vinogradov.
\newblock {\em {Selected work}}.
\newblock Berlin-New York: Springer-Verlag, 1985.

\bibitem{volostnov2018double}
A.~S. Volostnov.
\newblock {On Double Sums with Multiplicative Characters.}
\newblock {\em Mathematical Notes}, 104(3):197--203, 2018.

\bibitem{sh_vyugin_subgroups}
I.~V. Vyugin and I.~D. Shkredov.
\newblock {On additive shifts of multiplicative subgroups}.
\newblock {\em Matematicheskii Sbornik}, 203(6):81--100, 2012.

\end{thebibliography}

\end{document}